\RequirePackage{fix-cm}
\documentclass[smallextended]{svjour3}       

\smartqed  

\usepackage[utf8]{inputenc}
\DeclareUnicodeCharacter{00A0}{ }
\usepackage{graphicx}
\usepackage{amsmath, amscd, amsfonts, amssymb, graphicx, color}
\newcommand \s{^{*}}
\newcommand \+{^{\dagger}}
\begin{document}
\title{Generalized Principal Pivot Transform and its\\ Inheritance Properties
}
\subtitle{}


\author{	K. Kamaraj        \and
	P. Sam Johnson \and Sachin Manjunath Naik 
}


\institute{	K. Kamaraj   \at
	Department of Mathematics, University College of Engineering Arni,\\ Anna University, Arni 632326, India. \\
	\email{krajkj@yahoo.com}           
	\and
	P. Sam Johnson \at
	Department of Mathematical and Computational Sciences,\\ National Institute of Technology Karnataka (NITK), Surathkal, 		Mangaluru 575 025, India  \\
	\email{sam@nitk.edu.in} 
	\and
	Sachin Manjunath Naik \at
	Department of Mathematical and Computational Sciences,\\ National Institute of Technology Karnataka (NITK), Surathkal, 		Mangaluru 575 025, India  \\
	\email{sachinmaths46@gmail.com} 
}

\date{Received: date / Accepted: date}

\maketitle

\begin{abstract}
	In this paper, some more properties of the generalized principal pivot transform are derived.
	Necessary and sufficient conditions for the equality between Moore-Penrose inverse of a generalized principal pivot transform and its complementary generalized principal pivot transform are presented. It  has been shown that the generalized principal pivot transform preserves the rank of symmetric part of a given square matrix. These results appear to be more generalized than the existing ones. Inheritance property of $P_{\dagger}$-matrix are also characterized for generalized principal pivot transform.
\end{abstract}

	\section{Introduction}
	Let $ M $ be an $ n\times n $ complex matrix partitioned into blocks as $\left(\begin{array}{cc}A & B \\ C & D \\ \end{array}\right)$ where $ A $ is an invertible matrix.  The principal pivot transform of $ M $ with respect to $ A $ is defined as $\widetilde{M}=\left(\begin{array}{cc}A^{-1} & -A^{-1}B \\ CA^{-1} & S \\ \end{array}\right)$, where $ S=D-CA^{-1}B $ is the Schur complement of $ A $ in $ M$.  The operation that transforms $ M\mapsto \widetilde{M} $ is called the principal pivot transform of $ M $ with respect to $A$, denoted by $ ppt(M, A)$.  Properties and applications of the principal pivot transform with an interesting history are found in the elegant papers \cite{bishtravindranKCS,tsatsomeros}.
	
 AR. Meenakshi \cite{meenakshi} was perhaps the first to study the generalized principal pivot transform in the context of finding relationship between the generalized principal pivot transform and range-Hermitian matrices.
	Rajesh Kannan and Bapat \cite{rajesh_first_paper,rajeshkannan} defined generalized principal pivot transform and discussed its properties.

	In this paper, we  derive some characterizations on generalized principal pivot transform of complex partitioned matrices of the form
	$M=\left(\begin{array}{cc}
	A & B \\ C & D \\ \end{array}\right)$.  In Section 2, we state some definitions and results which are useful in the sequel.  In Section 3, we give some necessary and sufficient conditions to express the Moore-Penrose inverse of $ M $ in terms of generalized principal pivot transform of a suitable matrix. Few necessary conditions are given in \cite{rajeshkannan} for preserving  symmetric part of the matrix by generalized principal pivot transform. We prove the results with weaker assumptions. We also prove the domain-range exchange property for a larger class of matrices using generalized principal pivot transform. In the concluding section, we discuss inheritance properties of the generalized principal pivot transform of $P_{\+}$-matrices which are relevant and useful in the context of the linear complementarity problem.

	\section{Notations, Definitions and Preliminary Results}
	Throughout this paper, we shall deal with $ \mathbb C^{m\times n} $, the space of $ m\times n $ complex matrices.  For any $A\in \mathbb C^{m\times n}$, let $A^*$, $R(A)$, $N(A)$ and $rank(A)$ denote the complex conjugate transpose, range space, null space and rank of $A$, respectively.
	The Moore-Penrose inverse of $A\in  \mathbb C^{m\times n}$, denoted by $A^\dag$ is the unique solution $ X\in \mathbb C^{n\times m}$  of the equations : $AXA=A$, $XAX=X$, $(AX)^*=AX$ and $(XA)^*=XA$. 
	
	If $X$ satisfies the first equation, then $X$ is called an $ \{1\} $-inverse of $ A $ and is denoted by $A^{(1)}$. The set of all $ \{1\} $-inverses of $ A $ is denoted by $ A\{1\}$.  In a similar way, we denote the sets by $ A\{1, 2\}$ and $ A\{1,2,3\}$.  Note that $ A\{1\} $ is non-empty.  A matrix $A\in \mathbb C^{n\times n} $  is said to be range-Hermitian if $R(A)=R(A^*)$. An easy consequence of the definition gives that $A$ is range-Hermitian if and only if $AA^\dag = A^\dag A$ \cite{benisraelbook}.
	
	\begin{definition}\cite{rajesh_first_paper}
		Let $M=\left(\begin{array}{cc}A & B \\ C & D \\ \end{array}\right)$ be a complex partitioned matrix. Then the generalized principal pivot transform of $M$ with respect to $A$ is defined by $gppt(M,A)=\left(\begin{array}{cc}A^\dag & -A^\dag B \\ CA^\dag & D-CA^\dag B \\ \end{array}\right)$. Similarly, the generalized principal pivot transform of $M$ with respect to $D$ is defined by
		$gppt(M,D)=\left(\begin{array}{cc}A-BD^\dag C &  BD^\dag \\ -D^\dag C & D^\dag  \\ \end{array}\right)$. Here $D-CA^\dag B$ and $A-BD^\dag C$ are called the generalized Schur complements of $M$ with respect to $A$ and $D$ respectively.
	\end{definition}
	
	The following theorem is well-known and quite useful in the sequel.
	\begin{theorem}\cite{benisraelbook}\label{solution}
		The system of equation $AXB=C$ is consistent if and only if $AA^{(1)} CB^{(1)} B =C$, for any $A^{(1)}\in A\{1\}$ and $B^{(1)}\in B\{1\}$. In this case, the general solution is  $$X=A^{(1)} C B^{(1)} + Z-A^{(1)} AZBB^{(1)} $$ where $Z$ is an arbitrary matrix. In particular, if $Z=0$, then $X=A^{(1)}CB^{(1)}$.
	\end{theorem}
	The following result is given in \cite{rajeshkannan} which is in general not true as illustrated in the example given after the statement of the theorem.
	\begin{theorem}[\cite{rajeshkannan}, Theorem 3.3]\label{rajeshkannan3.3}
		Let $M=\left(\begin{array}{cc}A & B \\ C & D \\ \end{array}\right)$ be a complex partitioned matrix such that $N(D^*)\subseteq N(C^*)$ and $N(A^*)\subseteq N(B^*)$. Then $gppt(M,A)^\dag $ $=$ $gppt(M,D)$.
	\end{theorem}

	\begin{example}
		Let $M=\left(\begin{array}{cc}0 & 0 \\ 1 & 1 \\ \end{array}\right)$ with $A=B=0$ and $C=D=1$. Then by an easy computation, we can show that $gppt(M,A)=\left(\begin{array}{cc}0 & 0 \\ 0 & 1 \\ \end{array}\right)$ and $gppt(M,D)=\left(\begin{array}{cc}0 & 0 \\ -1 & 1 \\ \end{array}\right)$. Clearly, $N(D^*)=N(C^*)$ and $N(A^*)=N(B^*)$. In addition to the above, $gppt(M,A)=gppt(M,A)^\dag$. But  $gppt(M,A)^\dag\neq gppt(M,D)$.
	\end{example}
	In fact, a revised version of the Theorem \ref{rajeshkannan3.3}  was  proved in \cite{bishtravindranKCS} and it is given below for the sake of completeness.
	
	\begin{theorem}[\cite{bishtravindranKCS}, Theorem 4.1]\label{mainsufficient}
		Let $M=\left(\begin{array}{cc}A & B \\ C & D \\ \end{array}\right)$
		be a complex partitioned matrix such that
		 $N(A)\subseteq N(C)$, $N(A^*)\subseteq N(B^*)$, $N(D)\subseteq N(B)$ and $N(D^*)\subseteq N(C^*)$. Then $gppt(M,A)^\dag = gppt(M,D)$.
	\end{theorem}

	Note that the above theorem has only given a sufficient condition. But one of the objectives of this paper is to give  sufficient as well as necessary conditions for the complex partitioned matrix $M=\left(\begin{array}{cc}A & B \\ C & D \\ \end{array}\right)$ to have $gppt(M,A)^\dag = gppt(M,D)$. Another objective is to generalize the rank condition in the following theorem by dropping null-space inclusions in the hypothesis.
	
		\begin{theorem}[\cite{rajeshkannan}, Theorem 4.2]\label{rajeshkannan4.2}
		Let $M=\left(\begin{array}{cc}A & B \\ C & D \\ \end{array}\right)$ be a complex partitioned matrix such that $N(A)\subseteq N(C)$ and $N(A^*)\subseteq N(B^*)$. If $A$ is range-Hermitian, then $rank(S(gppt(M,A)))=rank(S(M))$.
	\end{theorem}

	Note that the symmetric part of  $A\in \mathbb C^{n\times n} $ is defined by $ \frac{A+A^*}{2} $ and is denoted by $ S(A)$.  A matrix $A\in \mathbb C^{n\times n} $ is said to be almost skew-Hermitian if $rank(S(A))=1$. Properties of skew-Hermitian matrices can be found in \cite{projeshKCS,mcdonald} and the references cited therein.

\begin{theorem}[\cite{projeshKCS}, Theorem 4.1]\label{projeshkcs}
	Let $A$ be a square matrix. Then $A$ is almost skew-Hermitian if and only if $A^\dag$ is almost skew-Hermitian.
\end{theorem}

			\begin{definition}[\cite{rajeshkannan2}]
			A real $n \times n$ square matrix $A$ is said to be a $P_{\+}$-matrix 
			if for each non zero $x\in R(A\s)$ there is an $i \in \{1, 2, \ldots, n\}$ such that $(x)_i(Ax)_i >0$.
			Equivalently, for any $x\in R(A\s)$ the inequalities $(x)_i(Ax)_i \leq 0$ for all $i=1, 2, \ldots, n$ imply that $x=0$.  It is proved in \cite{rajeshkannan2}
			that a real matrix $A$ is a $P_{\+}$-matrix if and only if $A\+$ is a
			$P_{\+}$-matrix.
		\end{definition}

	Given $Q \in \mathbb{R}^{n\times n}$ and $q \in \mathbb{R}^{n\times 1}$, the linear complementarity problem denoted by $LCP(Q,q)$ is to determine if there exists  $x \in \mathbb{R}^n$ such that $x\geq 0$, $y=Qx+q\geq 0$ and $x\s y=0$.
	
	\begin{definition}[\cite{MR3093070}]
		Let $M\in \mathbb{R}^{n \times n}$. Then $M$ is called a $R_\dagger$-matrix  if the solution for $LCP(M,0)$ in $R(M\s)$ is the zero solution.
	\end{definition}

		The following results in \cite{bishtravindranKCS} are used in the section to discuss inheritance properties of the generalized principal pivot transform.
			\begin{theorem}[\cite{bishtravindranKCS}, Theorem 3.1] \label{t31}
			Let $M=\left(\begin{array}{cc}A & B \\ C & D \\ \end{array}\right)$ be a real matrix and $F=D-CA\+B$. Then $R(C\s)\subseteq R(A\s)$, $R(B)\subseteq R(A)$, $R(C)\subseteq R(F)$ and $R(B\s)\subseteq R(F\s)$ if and only if $M\+=\left(\begin{array}{cc}A\+ +A\+ BF\+ CA\+ & -A\+B F\+ \\ -F\+ CA\+ & F\+ \\ \end{array}\right)$.
			
		\end{theorem}
		
		\begin{theorem}[\cite{bishtravindranKCS}, Theorem 3.2]
			Let $M=\left(\begin{array}{cc}A & B \\ C & D \\ \end{array}\right)$ be a real matrix and $G=A-BD\+C$. Then $R(B\s)\subseteq R(D\s)$, $R(C)\subseteq R(D)$, $R(B)\subseteq R(G)$ and $R(C\s)\subseteq R(G\s)$ if and only if $M\+=\left(\begin{array}{cc}G\+ & -G\+B D\+ \\ -D\+ CG\+ & D\+ + D\+CG\+ BD\+ \\ \end{array}\right)$.
			
		\end{theorem}

We provide a general version of the following theorems in the paper.
		
		\begin{theorem}[\cite{bishtravindranKCS}, Theorem 5.10]\label{t510}
			Let $M=\begin{pmatrix}A& B\\ C& D\end{pmatrix}$ with $A$, $B$, $C$ and $D$ be square matrices of same order satisfying $R(B) \subseteq R(A)$, $R(C\s)\subseteq R(A\s)$, $R(C) \subseteq R(F)$ and $R(B\s) \subseteq R(F\s)$. If $M$ is a $R_\dagger$-matrix, then $gppt(M,A)$ is a $R_\dagger$-matrix.
			
		\end{theorem}
		
		\begin{theorem}[\cite{bishtravindranKCS}, Theorem 5.11]\label{t511}
			Let $M=\begin{pmatrix}A& B\\ C& D\end{pmatrix}$ with $A$, $B$, $C$ and $D$ be square matrices of same order satisfying $R(B\s) \subseteq R(D\s)$, $R(C)\subseteq R(D)$, $R(B) \subseteq R(G)$ and $R(C\s) \subseteq R(G\s)$. If $M$ is a $R_\dagger$-matrix, then $gppt(M,D)$ is a $R_\dagger$-matrix.
			
\end{theorem}

	\section{Characterizations and Properties}
	 We start with some characterizations  to express the Moore-Penrose inverse of $ M $ in terms of generalized principal pivot transform of a suitable matrix.

	\begin{theorem}\label{main1}
		Let $M=\left(\begin{array}{cc}A & B \\ C & D \\ \end{array}\right)$. Then the following are equivalent :
		\begin{enumerate}
			\item $gppt(M,A)^\dag = gppt(M,D)$ ;
			\item $CA^\dag A = DD^\dag C$ and  $AA^\dag B = BD^\dag D$ ;
			\item $N(A)\subseteq N(D^*C)$, $N(A^*)\subseteq N(DB^*)$, $N(D)\subseteq N(A^*B)$ and $N(D^*)\subseteq N(AC^*)$.
		\end{enumerate}
	\end{theorem}
	\begin{proof}
		(1)$ \iff $ (2) :    Let $$P =gppt(M,A)= \left(\begin{array}{cc}A^\dag & -A^\dag B \\ CA^\dag & D-CA^\dag B \\ \end{array}\right)$$ and $$Q= gppt(M,D)=\left(\begin{array}{cc}A-BD^\dag C &  BD^\dag \\ -D^\dag C & D^\dag  \\ \end{array}\right).$$   Then $$PQ = \left(\begin{array}{cc}A^\dag A & 0 \\ CA^\dag A-DD^\dag C & DD^\dag \\ \end{array}\right)$$ and $$QP = \left(\begin{array}{cc}A A^\dag & BD^\dag D-AA^\dag B \\ 0 & D^\dag D \\ \end{array}\right).$$ Thus $PQ$ and $QP$ are Hermitian if and only if $CA^\dag A = DD^\dag C$ and $AA^\dag B = BD^\dag D$. It is also easy to verify that $Q\in P\{1,2\}$. Therefore $P^\dag = Q$ if and only if $CA^\dag A = DD^\dag C$ and $AA^\dag B = BD^\dag D$.

		(2) $\implies $ (3) : Suppose that $CA^\dag A = DD^\dag C$ and  $AA^\dag B = BD^\dag D$. Then pre-multiplying by $D^\dag$ and post-multiplying by $A^\dag $ in the first equation, we get $D^\dag CA^\dag A = D^\dag C$ and  $CA^\dag  = DD^\dag CA^\dag $, respectively. Thus $N(A)\subseteq N(D^*C)$ and $N(D^*)\subseteq N(AC^*)$. Similarly, if we pre-multiply by $A^\dag $ and post-multiply by $D^\dag $ in the second equation, we will get the other two inclusions.
		
		(3) $\implies$ (2) :
		Suppose that $N(A)\subseteq N(D^*C)$, $N(A^*)\subseteq N(DB^*)$, $N(D)\subseteq N(A^*B)$ and $N(D^*)\subseteq N(AC^*)$. It is straightforward to prove that $N(D^*C)=N(D^\dag C)$ and $N(AC^*)=N((A^\dag)^*C^*)$. Thus $N(A)\subseteq N(D^\dag C)$ and $N(D^*)\subseteq N((A^\dag )^*C^*)$ which are equivalent to $D^\dag CA^\dag A=D^\dag C$ and $DD^\dag CA^\dag = CA^\dag $ respectively. Pre-multiply by $D$ in the first equation and using the second equation we get that $CA^\dag A=DD^\dag C$. Similarly, we can show the other relation. This completes the proof of the theorem.
	\end{proof}

	The following example shows that the assumptions in  Theorem \ref{main1} are relatively weaker than the ones given in Theorem \ref{mainsufficient}.
	\begin{example}
		Let $M=\left(\begin{array}{cc}A & B \\C & D \\ \end{array} \right) $ with $A=D=\left(\begin{array}{cc}1 & 1 \\1 & 1 \\ \end{array} \right) $ and $B=C=\left(\begin{array}{cc}0 & 1 \\1 & 0 \\ \end{array} \right) $. Then $A^\dag = D^\dag= \frac{1}{4}A$. Hence $CA^\dag A = DD^\dag C$ and $AA^\dag B=BD^\dag D$ but $N(A)\nsubseteq N(C)$, $N(A^*)\nsubseteq N(B^*)$, $N(D)\nsubseteq N(B)$ and $N(D^*)\nsubseteq N(C^*)$. Further,
		$$gppt(M,A) =\frac{1}{4}\left(\begin{array}{cc}A & -A \\A & 3A \\ \end{array} \right)$$ and $$gppt(M,A)^\dag =\frac{1}{4}\left(\begin{array}{cc}3A & A \\-A & A \\ \end{array} \right) = gppt(M,D).$$ \end{example}
	
	If the generalized principal pivot transform with respect to suitable matrices is applied twice to a complex partitioned matrix $ M $, we have found some  conditions under which one would get back to the same matrix $ M $. One way of implications of the following result has already been proved by Bisht et. al. (\cite{bishtravindranKCS}, Lemma 4.1). We shall now prove that the sufficient conditions given in the said results are necessary as well.
	
	\begin{theorem}\label{mformula}
		Let $M=\left(\begin{array}{cc}A & B \\C & D \\ \end{array} \right) $. Then the following statements hold true.
		\begin{enumerate}
			\item $gppt(gppt(M,A),A^\dag )=M$ if and only if $N(A)\subseteq N(C)$ and $N(A^*)\subseteq N(B^*)$.
			\item $gppt(gppt(M,D),D^\dag )=M$ if and only if $N(D)\subseteq N(B)$ and $N(D^*)\subseteq N(C^*)$.
		\end{enumerate}
	\end{theorem}
	\begin{proof}
		Let $P=\left(\begin{array}{cc}A^\dag & -A^\dag B \\ CA^\dag & D-CA^\dag B \\ \end{array}\right)$. Then we have $$gppt(gppt(M,A),A^\dag )= gppt(P,A^\dag)= \left(\begin{array}{cc}A & AA^\dag B \\ CA^\dag A & D \\ \end{array}\right).$$ Thus $gppt(P,A^\dag )=M$ if and only if $N(A)\subseteq N(C)$ and $N(A^*)\subseteq N(B^*)$. Similarly, we can also prove the other one. This completes the proof.
	\end{proof}
	The next result is an immediate consequence of Theorems \ref{main1} and \ref{mformula}.
	\begin{theorem}
		Let $M=\left(\begin{array}{cc}A & B \\C & D \\ \end{array} \right) $, $F=D-CA^\dag B$ and $E=A-BD^\dag C$. If $P=gppt(M,A)$ and $Q=gppt(M,D)$, then the folowing statements hold true.
		\begin{enumerate}
			\item $M^\dag = gppt(P,F)$ if and only if $N(A)\subseteq N(C)$, $N(A^*)\subseteq N(B^*)$, $N(F)\subseteq N(B)$ and $N(F^*)\subseteq N(C^*)$.
			\item $M^\dag = gppt(Q,E)$ if and only if $N(D)\subseteq N(B)$, $N(D^*)\subseteq N(C^*)$, $N(E)\subseteq N(C)$ and $N(E^*)\subseteq N(B^*)$.
		\end{enumerate}
	\end{theorem}
	\begin{proof}
		By Theorem \ref{mformula}, $gppt(P,A^\dag)=M$ if and only if $N(A)\subseteq N(C)$ and $N(A^*)\subseteq N(B^*)$. Thus $M^\dag =gppt(P,A^\dag)^\dag$ if and only if $N(A)\subseteq N(C)$ and $N(A^*)\subseteq N(B^*)$. With the assumptions $N(A)\subseteq N(C)$ and $N(A^*)\subseteq N(B^*)$, by Theorem \ref{main1}, we have $M^\dag =gppt(P,A^\dag)^\dag = gppt(P, F)$ if and only if $CA^\dag =FF^\dag CA^\dag $ and $A^\dag B=A^\dag B F^\dag F$. By applying the fact that $AA^\dag B= B$ and $ CA^\dag A = C$, we will get that $M^\dag = gppt(P, F)$ if and only if $C =FF^\dag C $ and $ B= B F^\dag F$. It concludes that $M^\dag = gppt(P,F)$ if and only if $N(A)\subseteq N(C)$, $N(A^*)\subseteq N(B^*)$, $N(F)\subseteq N(B)$ and $N(F^*)\subseteq N(C^*)$. The proof of other part is quite similar.
	\end{proof}
	\begin{lemma}\label{x}
		Let $P=gppt(M,A)$ and $Q=gppt(M,D)$. If $X=\left(\begin{array}{cc}A & B \\ 0 & I \\ \end{array}\right)$, $Y=\left(\begin{array}{cc}I & 0 \\ C & D \\ \end{array}\right)$, $Z=\left(\begin{array}{cc}A^\dag &  -A^\dag B\\ 0 & I \\ \end{array}\right)$ and $\widehat{Z}=\left(\begin{array}{cc}I &  0\\ -D^\dag C & D^\dag \\ \end{array}\right)$, then $Z\in X\{1,2,4\}$ and $\widehat{Z}\in Y\{1,2,3\}$. Moreover, $YZ=P$ and $X\widehat{Z}=Q$. 	In particular, $Z=X^\dag$ if and only if $N(A^*)\subseteq N(B^*)$ and $\widehat{Z}=Y^\dag$ if and only if $N(A)\subseteq N(C)$.
	\end{lemma}
	\begin{proof}
		It can be proved from straightforward computations.
	\end{proof}

	The following theorem is a generalization of Theorem \ref{rajeshkannan4.2}, which shows that the generalized principal pivot transform preserves the rank of symmetric part of the matrix.

	\begin{theorem}\label{t1}
		Let $M= \begin{pmatrix}
		A&B\\C&D
		\end{pmatrix}$ be a partitioned square matrix with $A$ and $D$ are square matrices. If $A$ is range-Hermitian and $R(B+C\s )\subseteq R(A )$, then $rank(S(M))=rank(S(gppt(M,A))) $.
	\end{theorem}
	\begin{proof}
		Let $X= \begin{pmatrix}
		A&B\\
		0&I
		\end{pmatrix}$
		and $P= gppt(M,A)= \begin{pmatrix}
		A\+ & -A\+B\\
		CA\+ & D-CA\+B
		\end{pmatrix}$.  Then by easy computation, we can prove that 
		$$X\s PX=\begin{pmatrix}
		A\s A\+ A & 0\\
		B\s A\+A+CA\+A&D
		\end{pmatrix}.$$ Since $A$ is range-Hermitian, $A\s A\+A=A\s AA\+=A\s $ and $R(B+C\s )\subseteq R(A\s )$ give that $A\+A(B+C\s )=B +C\s .$ Equivalently, $B\s A\+A+CA\+A=B\s +C.$ Thus $X\s PX=\begin{pmatrix}
		A\s & 0\\
		B\s +C&D
		\end{pmatrix}$. Also, $X\s (P+P\s)X=M+M\s .$ This equality shows that $rank(M+M\s )\leq rank(P+P\s )$. Now, set $Y=\begin{pmatrix}
		A\+ & -A\+B\\
		0& I
		\end{pmatrix}$. Again by simple calculation and using the fact that $AA\+=A\+A$, we get $Y\s MY=\begin{pmatrix}
		(A\+) \s & 0\\
		CA\+-B\s (A\+) \s & D-CA\+B
		\end{pmatrix}$. Thus $Y\s (M+M\s)Y=P+P\s $. It ensures the rank equality as $rank(P+P\s )=rank(M+M\s )$.
	\end{proof}
	\begin{remark}
		The assumptions given in Theorem 
		\ref{t1}		 are weaker than the ones in Theorem \ref{rajeshkannan4.2}. It is shown by the following example : Let $A=\begin{pmatrix}
		1 & 1\\
		1 & 1
		\end{pmatrix}$, $B=\begin{pmatrix}
		0\\
		1
		\end{pmatrix}$ and $C=\begin{pmatrix}
		1 & 0
		\end{pmatrix}$. Then $B+C\s =\begin{pmatrix}
		1\\
		1
		\end{pmatrix}$, thus $R(B+C\s )=R(A\s)$. But $N(A)\not \subseteq N(C)$ and $N(A^*)  \not \subseteq N(B^*)$.
	\end{remark}
	
	\begin{corollary}
		Let $M=\begin{pmatrix}
		A&B\\
		C&D
		\end{pmatrix}$ be a square matrix with $A$ as a range-Hermitian matrix. If $R(B+C\s )\subseteq R(A )$, then the following are equivalent :
		\begin{itemize}
			\item [{(i)}] $M$ is almost skew-symmetric.
			\item [{(ii)}]$M\+$  is almost skew-symmetric.
			\item [{(iii)}] $gppt(M,A)$  is almost skew-symmetric.
			\item [{(iv)}]  $gppt(M,A)\+$ is almost skew-symmetric.
			
		\end{itemize}
	\end{corollary}
	\begin{proof}
		$	(i)\Leftrightarrow (iii)$ follow from Theorem \ref{t1} and other implications follow from Theorem \ref{projeshkcs}.
	\end{proof}

	\begin{theorem}\label{t2}
		Let $M= \begin{pmatrix}
		A&B\\C&D
		\end{pmatrix}$ be a partitioned square matrix with $A$ and $D$ are square matrices. If $D$ is range-Hermitian and $R(C+B\s )\subseteq R(D )$, then $rank(S(M)) = rank(S(gppt(M,D)))$.
	\end{theorem}
	\begin{proof}The proof is similar to Theorem \ref{t1} by setting 
		$ X=\left(\begin{array}{cc}I &  0\\ C & D \\ \end{array}\right)$ and 
		$Y=\left(\begin{array}{cc}I &  0\\ -D\+ C & D\+ \\ \end{array}\right)$.
	\end{proof}

	\begin{corollary}
		Let $M=\begin{pmatrix}
		A&B\\
		C&D
		\end{pmatrix}$ be a square matrix with $D$ as a range-Hermitian matrix. If $R(C+B\s )\subseteq R(D )$, then the following are equivalent :
		\begin{itemize}
			\item [{(i)}] $M$ is almost skew-symmetric.
			\item [{(ii)}]$M\+$  is almost skew-symmetric.
			\item [{(iii)}] $gppt(M,D)$  is almost skew-symmetric.
			\item [{(iv)}]  $gppt(M,D)\+$ is almost skew-symmetric.
			
		\end{itemize}
	\end{corollary}
	\begin{proof}
		$	(i)\Leftrightarrow (iii)$ follow from Theorem \ref{t2} and other implications follow from Theorem \ref{projeshkcs}.
	\end{proof}
	
	We now discuss domain-range exchange property which  is well established for nonsingular matrices. Bishat et. al. \cite{bishtravindranKCS} extended the domain range property for the singular matrices with some assumptions (\cite{bishtravindranKCS}, Lemma 4.2). But we have explored the same results for a larger class of matrices.

\begin{theorem}
		Let $M=\begin{pmatrix}
		A&B\\
		C&D
		\end{pmatrix}$. If $CA\+A=DD\+C$ and $AA\+B=BD\+D$, then the following are equivalent :
		\begin{itemize}
			\item [{(i)}]$P=gppt(M,A)$ is range-Hermitian.
			\item [{(ii)}]$Q=gppt(M,D)$ is range-Hermitian.
			\item [{(iii)}] $A$ and $D$ are range-Hermitian.
		\end{itemize}	
	\end{theorem}
	\begin{proof}
		From the proof of Theorem \ref{main1}, we  observe that 
		$$
		PQ=\begin{pmatrix}
		A\+A & 0\\
		CA\+A-DD\+C &DD\+ 
		\end{pmatrix} \quad\text{and}\quad QP=\begin{pmatrix}
		AA\+ &AA\+B-BD\+D \\
		0 & D\+D 
		\end{pmatrix}.$$
		$(iii)\Rightarrow (i):$ By Theorem \ref{main1}, we have $P\+=Q.$ Hence $PP\+=PQ$. Since $A$ and $D$ are range-Hermitian, we have $PQ=QP$. Therefore $PP\+=PQ=QP=P\+P$. Hence $P$ is range-Hermitian.\\
		$(i)\Leftrightarrow (ii)$: By Theorem \ref{main1}, we have $Q\+=P$. Also observing the fact that $A$ is range-Hermitian, then so is $A\+$.\\
		$(i)\Rightarrow (iii)$: 
		Set  $X= \begin{pmatrix}
		I & 0\\
		C & I
		\end{pmatrix}$, $Y= \begin{pmatrix}
		I & -B\\
		0 & I
		\end{pmatrix}$ and $Z= \begin{pmatrix}
		A\+ & 0\\
		0 & D
		\end{pmatrix}$. Then  $X^{-1}= \begin{pmatrix}
		I & 0\\
		-C & 0
		\end{pmatrix}$ and $Y^{-1}= \begin{pmatrix}
		I & B\\
		0 & I
		\end{pmatrix}$. Also it is easy to verify that $P=XZY$ and $Q=Y^{-1}Z\+X^{-1}$. Moreover, $Q \in P\{1\}$ and $P\in Q\
		\{1\}$. Suppose $P$ is range-Hermitian. Then $R(P)=R(P\s )$. It concludes that $PP^{(1)}P\s =P\s $ for any $P^{(1)} \in P\{1\}$. Thus $PQP\s=P\s.$ $PQP\s =XZZ\+X^{-1}P\s =\begin{pmatrix}
		I & 0\\
		C & I
		\end{pmatrix}\begin{pmatrix}
		A\+A & 0\\
		0 & DD\+
		\end{pmatrix}\begin{pmatrix}
		I & 0\\
		-C & I
		\end{pmatrix}P\s\\$
		\begin{equation}  \label{e1}
		=\begin{pmatrix}
		A\+A(A\+) \s & A\+A(A\+) \s C\s\\
		CA\+A(A\+) \s -DD\+C(A\+) \s  & CA\+A(A\+)\s C\s -DD\+C(A\+) \s C\s\\
		-DD\+B\s (A\+) \s & DD\+D\s -DD\+B\s (A\+) \s C\s
		\end{pmatrix}=P\s 
		\end{equation}\\
		$\Rightarrow$ $A\+A(A\+) \s =(A\+) \s $ $\Rightarrow$ $R((A\+) \s )\subseteq R(A\+A)=R(A\s )$. This shows that $R(A\s )=R(A)$, hence $A$ is range-Hermitian. Also using the assumption  $CA\+A=DD\+C$ and $AA\+B=BD\+D$, the equation 
		(\ref{e1}) reduces to $$\begin{pmatrix}
		(A\+) \s & (A\+) \s C\s\\
		-DD\+B\s (A\+) \s & DD\+D\s -DD\+B\s (A\+) \s C\s
		\end{pmatrix}=P\s $$ which implies that 
		$-DD\+B\s (A\+) \s =-B\s (A\+) \s $ and thus $DD\+D\s =D\s $. This gives $R(D)=R(D\s )$, hence $D$ is range-Hermitian. This completes the proof.
	\end{proof}
	
	\begin{example}
		Let $M=\begin{pmatrix}
		0 & -2\\
		1 & 0
		\end{pmatrix}$ with $A=0,B=-2,C=1,D=0$. Then $ P=gppt(M,A)=\begin{pmatrix}
		0 & 0\\
		0 & 0
		\end{pmatrix}$. Then $rank(M+M\s )=1\ne 0=rank(P+P\s )$. This shows that assumption $R(B+C\s )\subseteq R(A\s )$ given in Theorem \ref{t1} is indispensable.
	\end{example}

		\begin{theorem}\label{t11}
		Let $M=\left(\begin{array}{cc}A & B \\ C & D \\ \end{array}\right)$ be a complex partitioned matrix.
		\begin{enumerate}
			\item If  $M\left(\begin{array}{c}A^\dag Ax_1  \\ x_2 \\ \end{array}\right)=\left(\begin{array}{c}y_1  \\ y_2 \\ \end{array}\right) $, then $gppt(M,A)\left(\begin{array}{c}y_1  \\ x_2 \\ \end{array}\right) = \left(\begin{array}{c}A^\dag Ax_1  \\ y_2 \\ \end{array}\right)$. Conversely, if $R(B)\subseteq R(A)$, then for any $y_1 \in R(A)$, $$gppt(M,A)\left(\begin{array}{c}y_1  \\ x_2 \\ \end{array}\right) = \left(\begin{array}{c}A^\dag Ax_1  \\ y_2 \\ \end{array}\right) \implies M\left(\begin{array}{c}A^\dag Ax_1  \\ x_2 \\ \end{array}\right)= \left(\begin{array}{c}y_1  \\ y_2 \\ \end{array}\right).$$
			\item  If  $M\left(\begin{array}{c}x_1  \\ D^\dag Dx_2 \\ \end{array}\right)=\left(\begin{array}{c}y_1  \\ y_2 \\ \end{array}\right) $, then $gppt(M,D)\left(\begin{array}{c}x_1  \\ y_2 \\ \end{array}\right) = \left(\begin{array}{c}y_1  \\ D^\dag Dx_2 \\ \end{array}\right)$. Conversely, if $R(C)\subseteq R(D)$, then for any $y_2 \in R(D)$, $$gppt(M,D)\left(\begin{array}{c}x_1  \\ y_2 \\ \end{array}\right) = \left(\begin{array}{c}y_1  \\ D^\dag Dx_2 \\ \end{array}\right) \implies M\left(\begin{array}{c}x_1  \\ D^\dag Dx_2 \\ \end{array}\right)= \left(\begin{array}{c}y_1  \\ y_2 \\ \end{array}\right).$$
		\end{enumerate}
	\end{theorem}

	\begin{proof}
		Suppose, $M\left(\begin{array}{c}A^\dag Ax_1  \\ x_2 \\ \end{array}\right)=\left(\begin{array}{c}y_1  \\ y_2 \\ \end{array}\right) $. Then $Ax_1=y_1-Bx_2$ and $CA^\dag Ax_1+Dx_2 = y_2$. Substitute the value of $Ax_1$ in the second equation, we get $$CA^\dag y_1 +(D-CA^\dag B)x_2 =y_2.$$ Pre-multiply the equation by $A^\dag$, we  get $$A^\dag y_1 -A^\dag x_2 = A^\dag Ax_1.$$ Thus $gppt(M,A)\left(\begin{array}{c}y_1  \\ x_2 \\ \end{array}\right) = \left(\begin{array}{c}A^\dag Ax_1  \\ y_2 \\ \end{array}\right)$.
		
		Conversely, suppose that $gppt(M,A)\left(\begin{array}{c}y_1  \\ x_2 \\ \end{array}\right) = \left(\begin{array}{c}A^\dag Ax_1  \\ y_2 \\ \end{array}\right)$. Then $A^\dag y_1=A^\dag Ax_1+A^\dag Bx_2$ and $CA^\dag y_1+Dx_2-CA^\dag Bx_2 =y_2$. Substitute the value of $A^\dag y_1$ in the second equation we get $$CA^\dag Ax_1+Dx_2 = y_2.$$ Pre-multiply the first equation by $A$ we get $$AA^\dag y_1 - AA^\dag Bx_2 = Ax_1.$$ Then using the fact that $AA^\dag B=B$ and $AA^\dag y_1=y_1$,  we  get $$Ax_1+Bx_2=y_1.$$ Thus $M\left(\begin{array}{c}A^\dag Ax_1  \\ x_2 \\ \end{array}\right)= \left(\begin{array}{c}y_1  \\ y_2 \\ \end{array}\right).$
		The second part can be proved in a similar way.   This completes the proof.
	\end{proof}

	The following result is  Theorem 4.2
	in  \cite{bishtravindranKCS} whose proof is not complete as illustrated here :  Let $P=\left(\begin{array}{cc}I &  0\\ B\s {(A\s)} \+ & I \\ \end{array}\right)$, $Q=\left(\begin{array}{cc}I &  A\+ B\\0  & I \\ \end{array}\right)$
	and $N=\left(\begin{array}{cc}A\s A &  0\\ 0& 0 \\ \end{array}\right)$. In  \cite{bishtravindranKCS}, it is claimed that $Q^{-1}N^{(1)} P^{-1}=gppt(M\s M, A\s A)$ which is not true, where $N^{(1)}=\left(\begin{array}{cc}{(A\s A)\+} &  0\\ 0 & 0 \\ \end{array}\right)$. Here we have given the  complete proof.
	
	\begin{theorem}[\cite{bishtravindranKCS}, Theorem 4.2]
		Let $M=\left(A|B\right)$ be a partition matrix with $A\in {\mathbb{R}}^{n\times r}$ and $B \in {\mathbb{R}}^{n\times (n-r)}$. If $R(B) \subseteq R(A)$, then $gppt(M\s M, A\s A)$ is an $\{1\}$-inverse of $M\s M$.
	\end{theorem}

	\begin{proof}
		We have $M\s M=\left(\begin{array}{cc}A\s A &  A\s B\\ B\s A & B\s B \\ \end{array}\right)$. Let $P$, $N$, $Q$ and $N^{(1)}$ be as given above. Since $R(B) \subseteq R(A)$, we have $M\s M= PNQ$. Now $Q^{-1}N^{(1)}P^{-1}= N^{(1)}$ is an $\{1\}$-inverse of $M\s M$ and  by  Theorem \ref{solution}, $K= N^{(1)}+Z- N^{(1)}M\s MZM\s M  N^{(1)}$ is also an $\{1\}$-inverse of $M\s M$ for any matrix $Z$ of appropriate size. Taking $Z=\left(\begin{array}{cc}0 &  -A\+ B\\ B\s{(A\s)}\+ & 0 \\ \end{array}\right)$ we get $$K= N^{(1)}+Z- \left(\begin{array}{cc}-{A\s A}\+ A\s[-AA\+BB\s + BB\s{(A\+)}\s A\s] A{(A\s A)\+} &  0\\ 0 & 0 \\ \end{array}\right).$$ Since $R(B)\subseteq R(A)$ we have $AA\+B=B$ and $B\s {(A\+)}\s A\s = B\s$. 
		Therefore $K= N^{(1)}+Z = \left(\begin{array}{cc}{(A\s A)\+} &  -A\+ B\\B\s{(A\s)}\+ & 0 \\ \end{array}\right)= gppt(M\s M, A\s A)$.
	\end{proof}

	\section{Some Inheritance Properties for Generalized Principal Pivot Transform}
	
	In this section, we consider inheritance properties of the generalized principal pivot transform with  two  classes of matrices,  $P_{\+}$-matrices \cite{rajeshkannan2} and $R_{\+}$-matrices \cite{MR3093070}.  These classes of matrices are relevant and useful in the context of the linear complementarity problem.

	\begin{theorem}\label{t15}
		Let $M=\begin{pmatrix}A& B\\ C& D\end{pmatrix}$ be a real matrix with $A$ and $D$ are square matrices. 	Let $M_0=\begin{pmatrix}A& AA\+B\\ CA\+A& D\end{pmatrix}$ and 
		$F=D-CA\+ B$. Suppose $R(C) \subseteq R(F)$ and $R(B\s) \subseteq R(F\s)$. If $M_0$ is a $P_\dagger$- matrix, then $gppt(M,A)$, $A$, $D$ are $P_\dagger$- matrices.
		On the other hand, suppose $R(C) \subseteq R(D)$ and $R(B\s) \subseteq R(D\s).$ If $gppt(M,A)$ is a $P_{\+}$-matrix, then $M_0$, $A$ and $F$ are  $P_{\+}$-matrices.
	\end{theorem}
	
	\begin{proof}
		Let $H=gppt(M,A)$. Consider $z=\left(\begin{array}{c}z_1\\ z_2\end{array}\right) \in R(H\s)$. Then there exists $v=\left(\begin{array}{c}v_1\\ v_2\end{array}\right)$ such that $H\s v =z$. Then $z_1= (A^\dagger)\s v_1+ (CA^\dagger)\s v_2$ and $z_2=(-A\+ B)\s v_1 + F\s v_2$. This gives that $z_1 \in R(A)$. So $z_1= AA\+z_1$. Also since $R(B\s)\subseteq R(F\s)$ we have $z_2 \in R(F\s)$ and so $z_2=F\+ F z_2$. Let $w=\left(\begin{array}{c}w_1\\ w_2\end{array}\right) = Hz$. Then $w_1= A\+ z_1 - A\+ B z_2$ and $w_2 = CA\+ z_1 + Fz_2$. This clearly gives that $w_1 \in R(A\+)$ and so $w_1 = A\+ A w_1$. Therefore $Hz=H\left(\begin{array}{c}AA\+z_1\\ z_2\end{array}\right)= w=\left(\begin{array}{c}A\+ A w_1\\ w_2\end{array}\right)$. Observe that $gppt(H,A\+)=M_0$. Now by $1$ of Theorem \ref{t11} we have $M_0\left(\begin{array}{c}A\+Aw_1\\ z_2\end{array}\right) = \left(\begin{array}{c}AA\+z_1\\ w_2\end{array}\right) =\left(\begin{array}{c}z_1\\ w_2\end{array}\right).$ Let $x=\left(\begin{array}{c}w_1\\ z_2\end{array}\right)$.
		Now suppose $(z)_i (Hz)_i \leq 0$  for all $i$. Then $(z_1)_i (w_1)_i \leq 0$ for all $i$ and $(z_2)_j (w_2)j \leq 0$ for all $j$. Now observe that $(w_1)_i=(x)_i$ for all $i$ ; for some $s$, $w_1\in \mathbb{C}^s$, hence $(z_2)_j=(x)_{j+s}$  for all $j$.  Similarly, 		 $(z_1)_i=(M_0x)_i$ for all $i$ ; 
		for some $t$, $z_1\in \mathbb{C}^t$, hence
		$(w_2)_j=(M_0x)_{j+t}$ for all $j$. Therefore we have $(x)_i(M_0x)_i \leq 0$, for all $i$. 
		By Theorem \ref{t31} we have $(M_0)\+ M_0= \begin{pmatrix}A\+A& 0\\ 0& F\+F\end{pmatrix}$. Hence $(M_0)\+ M_0 x= \left(\begin{array}{c}A\+ Aw_1\\ F\+Fz_2\end{array}\right)= \left(\begin{array}{c}w_1\\ z_2\end{array}\right)=x$. Therefore $x \in R((M_0)\s)$. Since $M_0 $ is a $P_{\+}$-matrix we have $x=0$. That is, $w_1=0$ and $z_2=0$.
Since	$M_0 \left(\begin{array}{c}w_1\\ z_2\end{array}\right)= \left(\begin{array}{c}z_1\\ w_2\end{array}\right) $  $z_1 = 0$ and therefore we have $z=0$. Hence $gppt(M,A)$ is a $P_{\+}$-matrix.

 Now we show that $A$ is a $P_{\+}$-matrix. Let $x_1 \in R(A\s)$. Suppose $(x_1)_i(Ax_1)_i \leq 0$ for all $i$. Take $x= \left(\begin{array}{c}x_1\\ 0\end{array}\right)$. Then $(M_0)\+ M_0 x= \left(\begin{array}{c}A\+ Ax_1\\ F\+F0\end{array}\right)= x$. Therefore $x \in R((M_0)\s)$. Also $(x)_j (M_0 x)_j=\left(\begin{array}{c}(x_1)_j(Ax_1)_j\\ 0\end{array}\right) \leq 0$, for all $j.$ Since $M_0$ is a $P_{\+}$-matrix, we have $x=0$, hence $x_1=0$. Therefore $A$ is a $P_{\+}$-matrix. 		
 Finally we show that $D$ is a $P_{\+}$-matrix. Let $x_2 \in R(D\s)$ so that $x_2= D\+Dx_2$. Since $F=D-CA\+B$ and $R(B\s) \subseteq R(F\s)$, we have $R(D\s) \subseteq R(F\s)$. Therefore $x_2=F\+Fx_2$. Suppose $(x_2)_i(Dx_2)_i \leq 0$ for all $i$. Define $x= \left(\begin{array}{c}0\\ x_2\end{array}\right)$. Then $(M_0)\+ M_0 x= \left(\begin{array}{c}A\+ A0\\ F\+Fx_2\end{array}\right)= x$. Therefore $x \in R(M_0)\s$. Also $(x)_j (M_0 x)_j \leq 0$, for all $j.$  Since $M_0$ is a $P_{\+}$-matrix, we have $x=0$, hence $x_2=0$. Therefore $D$ is a $P_{\+}$-matrix.  
 The proof of other part follows  by observing
		$gppt(M_0, A)=gppt(M,A)=H$, $gppt(H, A\+)=M_0$ and applying the fact that 
		a real matrix $A$ is a $P_{\+}$-matrix if and only if $A\+$ is a
		$P_{\+}$-matrix.		
	\end{proof}
	\begin{corollary}
		Let $M=\begin{pmatrix}A& B\\ C& D\end{pmatrix}$ be a real matrix with $A$ and $D$ are square matrices. 	Let $M_0=\begin{pmatrix}A& AA\+B\\ CA\+A& D\end{pmatrix}$ and 
		$F=D-CA\+ B$. Suppose $R(C) \subseteq R(F)$, $R(B\s) \subseteq R(F\s)$,  $R(C) \subseteq R(D)$ and $R(B\s) \subseteq R(D\s).$ Then
		$M_0$ is a $P_\dagger$- matrix if and only if $gppt(M,A)$ is a $P_\dagger$- matrix.

	\end{corollary}
	A similar result holds for complementary generalized principal pivot transform. We state these result below and the proof is omitted.

	\begin{theorem}
		Let $M=\begin{pmatrix}A& B\\ C& D\end{pmatrix}$ be a real matrix with $A$ and $D$ are square matrices. 	Let $M_1=\begin{pmatrix}A& BD\+D\\ DD\+C& D\end{pmatrix}$ and 
		$G=A-BD\+C$. Suppose $R(B) \subseteq R(G)$ and $R(C\s) \subseteq R(G\s)$.
		If $M_1$ is a $P_\dagger$- matrix, then $gppt(M,D)$, $A$, $D$ are $P_\dagger$- matrices.
		On the other hand, suppose $R(B) \subseteq R(A)$ and $R(C\s) \subseteq R(A\s).$ If $gppt(M,D)$ is a $P_{\+}$-matrix, then $M_1$, $D$ and $G$ are  $P_{\+}$-matrices.
	\end{theorem}

	\begin{corollary}
		Let $M=\begin{pmatrix}A& B\\ C& D\end{pmatrix}$ be a real matrix with $A$ and $D$ are square matrices. 	Let $M_1=\begin{pmatrix}A& BD\+D\\ DD\+C& D\end{pmatrix}$ and 
		$G=A-BD\+C$. Suppose $R(B) \subseteq R(G)$, $R(C\s) \subseteq R(G\s)$,   $R(B) \subseteq R(A)$ and $R(C\s) \subseteq R(A\s).$ Then
		$M_1$ is a $P_\dagger$- matrix if and only if $gppt(M,D)$ is a $P_\dagger$- matrix.

	\end{corollary}

	If we drop the range space conditions in  Theorem \ref{t15}, then the result may not hold, as shown in the following example.
	\begin{example}
		
		Let $M=\left(\begin{array}{ccc}2& -2& 1\\
		2& -2& 1\\
		-1& 1& -0.5\end{array}\right)$	with $A=\left(\begin{array}{ccc}2& \hspace{0.1cm} -2\\
		2& \hspace{0.1cm} -2
		\end{array}\right)$, $B=\left(\begin{array}{ccc}1\\
		1
		\end{array}\right)$, $C=\left(\begin{array}{ccc}-1& \hspace{0.1cm} 1
		\end{array}\right)$, $D=(-0.5)$. Now $A\+ = \left(\begin{array}{ccc}\frac{1}{8}& \hspace{0.1cm} \frac{1}{8}\\
		
		\\
		\frac{-1}{8}& \hspace{0.1cm} \frac{-1}{8}
		\end{array}\right)$, $M_0 = M$ and $F=0$. Therefore $R(C) \nsubseteq R(F)$ and $R(B\s) \nsubseteq R(F\s)$. 
		Let $x \in R((M_0)\s)$. Then $x=\alpha (2, -2, 1)\s$, where $\alpha \in \mathbb{R}$. Suppose $(x)_i (Mx)_i \leq 0$ for $i=1, 2, 3.$ Then $18 (\alpha)^2 \leq 0$. This gives $\alpha =0$. Therefore $x=0$. So $M_0$ is a $P_{\+}$-matrix.\\
		Now $H=gppt(M,A)=\left(\begin{array}{ccc}0.125& 0.125& -0.25\\
		-0.125& -0.125& 0.25\\
		-0.25& -0.25& 0\end{array}\right)$. Let $y=(0, 0, -1)\s$. Then $0 \neq y\in R(H\s)$ and $(y)_i(Hy)_i \leq 0$ for all $i=1, 2, 3$ which show that $H$ is not a $P_{\+}$-matrix.
		Also $0\neq x_0= (1,-1)\s \in R(A\s)$ and $(x_0)_i( Ax_0)_i = 0$. Therefore $A$ is not a $P_{\+}$-matrix.
		Now take $x_1=(1)$, clearly $x_1 \in R(D\s)$. $(x_1)(Dx_1)=-0.5 \leq 0$. Therefore $D$ is not a $P_{\+}$-matrix. So conclusion of the first part of Theorem \ref{t15} does not hold.

	\end{example}
	
	\begin{example}\label{eg5}
Consider $M=\left(\begin{array}{ccc}0.125& 0.125& -0.25\\
		-0.125& -0.125& 0.25\\
		-0.25& -0.25& 0\end{array}\right)$ with $A=\left(\begin{array}{ccc}0.125& \hspace{0.1cm} 0.125\\
		-0.125& \hspace{0.1cm} -0.125
		\end{array}\right)$, $B=\left(\begin{array}{ccc}-0.25\\
		0.25
		\end{array}\right)$, $C=\left(\begin{array}{ccc}-0.25& \hspace{0.1cm} -0.25
		\end{array}\right)$, $D=(0)$.
		Now $gppt(M,A)=H=\left(\begin{array}{ccc}2& -2& 1\\
		2& -2& 1\\
		-1& 1& -0.5\end{array}\right)$ and $M_0=M$. Clearly $R(C) \nsubseteq R(D)$ and 
		$R(B\s) \nsubseteq R(D\s)$. By the previous example, $H$ is a $P_{\+}$-matrix but $M_0$ is not a $P_{\+}$-matrix. So conclusion of the second part of Theorem \ref{t15} does not hold. 	This example also  shows that converse of the first part Theorem \ref{t15} is not true in general.
		
	\end{example}

In \cite{bishtravindranKCS}, for a $R_\dagger$-matrix $M$, some sufficient conditions are given for $gppt(M,A)$ to be a $R_\dagger$-matrix. We now prove generalized versions of them in the following results. 
	
	\begin{theorem}\label{t99}
		Let $M=\begin{pmatrix}A& B\\ C& D\end{pmatrix}$ be a real matrix with $A$ and $D$ are square matrices. 	Let $M_0=\begin{pmatrix}A& AA\+B\\ CA\+A& D\end{pmatrix}$ and 
		$F=D-CA\+ B$. Suppose $R(C) \subseteq R(F)$ and $R(B\s) \subseteq R(F\s)$.
		If $M_0$ is a $R_\dagger$- matrix, then $gppt(M,A)$ is a $R_\dagger$- matrix. On the other hand, suppose $R(C) \subseteq R(D)$ and $R(B\s) \subseteq R(D\s).$ If $gppt(M,A)$ is a $R_{\+}$-matrix, then $M_0$ is a $R_{\+}$-matrix.
	\end{theorem}
	
	\begin{proof}
		Let $H=gppt(M,A)$. Let $z=\left(\begin{array}{c}z_1\\ z_2\end{array}\right) \in R(H\s)$ such that $z \geq 0$, $Hz= v=\left(\begin{array}{c}v_1\\ v_2\end{array}\right) \geq 0$ and $\langle v,z \rangle =0$. Then $v_1= (A^\dagger) z_1- (A^\dagger B) z_2$ and $v_2=(CA\+ ) z_1 + F z_2$. This gives that $v_1 \in R(A\s)$. So $v_1= A\+ Av_1$.  As $z \in R(H\s)$, there exists $u=\left(\begin{array}{c}u_1\\ u_2\end{array}\right)$ such that $H\s u=z$. Therefore $z_1= (A\+)\s u_1 + (CA\+)\s u_2$ and $z_2=(-A\+ B)\s u_1 +F\s u_2.$ This implies $z_1 \in R(A)$ and so $z_1 = AA\+ z_1$.  Since $R(B\s) \subseteq R(F\s)$, we obtain $z_2 \in R(F\s)$ and so $z_2=F\+ Fz_2$. Therefore $Hz=H\left(\begin{array}{c}AA\+z_1\\ z_2\end{array}\right)= v=\left(\begin{array}{c}A\+ A v_1\\ v_2\end{array}\right)$. Observe that $gppt(H,A\+)=M_0$. Now by $1$ of Theorem \ref{t11} we have $M_0\left(\begin{array}{c}A\+Av_1\\ z_2\end{array}\right) = \left(\begin{array}{c}AA\+z_1\\ v_2\end{array}\right) =\left(\begin{array}{c}z_1\\ v_2\end{array}\right).$ Let $x=\left(\begin{array}{c}v_1\\ z_2\end{array}\right)$ and $y=\left(\begin{array}{c}z_1\\ v_2\end{array}\right)$. We have $x\geq 0$,   $M_0x=y \geq 0$ and  $\langle x,y \rangle = \sum x_i y_i = \sum z_i v_i= \langle z,v \rangle =0$. 
Since $v_1 = A\+ A v_1$ and $z_2=F\+ Fz_2$,  by Theorem \ref{t31}, we have $(M_0)\+ M_0= \begin{pmatrix}A\+A& 0\\ 0& F\+F\end{pmatrix}$. Therefore $(M_0)\+ M_0 x= \left(\begin{array}{c}A\+ Av_1\\ F\+Fz_2\end{array}\right)= \left(\begin{array}{c}v_1\\ z_2\end{array}\right)=x$, hence  $x \in R((M_0)\s)$.
		Since $M_0 $ is a $R_{\+}$-matrix we have $x=0$, so  $v_1=0$ and $z_2=0$.
	As		$M_0 \left(\begin{array}{c}v_1\\ z_2\end{array}\right)= \left(\begin{array}{c}z_1\\ v_2\end{array}\right)$,  $z_1 = 0$ and hence $z=0$. Thus $gppt(M,A)$ is a $R_{\+}$-matrix. 		Now the proof of other part follows  by observing
		$gppt(M_0, A)=gppt(M,A)=H$ and $gppt(H, A\+)=M_0$.

	\end{proof}

	\begin{theorem}\label{end_thm}
		Let $M=\begin{pmatrix}A& B\\ C& D\end{pmatrix}$ be a real matrix with $A$ and $D$ are square matrices. 	Let $M_1=\begin{pmatrix}A& BD\+D\\ DD\+C& D\end{pmatrix}$ and 
		$G=A-BD\+C$. Suppose $R(B) \subseteq R(G)$ and $R(C\s) \subseteq R(G\s)$. 
		If $M_1$ is a $R_\dagger$- matrix, then $gppt(M,D)$ is a  $R_\dagger$-matrix.
				On the other hand, suppose $R(B) \subseteq R(A)$ and $R(C\s) \subseteq R(A\s).$ If $gppt(M,D)$ is a $R_{\+}$-matrix, then $M_1$ is a  $R_{\+}$-matrix.
	\end{theorem}
	
	\begin{proof}
The proof is similar to  Theorem \ref{t99}.
	\end{proof}
	
	\begin{center}
		\textbf{Acknowledgements}
	\end{center}
	
	The first author wishes to thank TNSCST, Government of Tamilnadu, India for the financial support through Young Scientist Fellowship and to carry out this work under collaborative Research Scheme.   The third author thanks the National
	Institute of Technology Karnataka (NITK), Surathkal for giving financial support

\end{document}